\documentclass[letterpaper, 10 pt, conference]{ieeeconf}  

\usepackage{ifthen}
\newboolean{showcomments}
\setboolean{showcomments}{true}
\usepackage[usenames,dvipsnames]{color}
\usepackage{todonotes}

\definecolor{bleudefrance}{rgb}{0.19, 0.55, 0.91}
\definecolor{ao(english)}{rgb}{0.0, 0.5, 0.0}

\newcommand{\addcite}[0]{\ifthenelse{\boolean{showcomments}}
{\textcolor{purple}{(add cite(s)) }}{}}%

\newcommand{\enrique}[1]{  \ifthenelse{\boolean{showcomments}}
{\todo[inline,color=bleudefrance]{Enrique: #1}}{}}
\newcommand{\emmargin}[1]{\ifthenelse{\boolean{showcomments}}{\marginpar{\color{bleudefrance}\tiny EM: #1}}{}}

\newboolean{showedits}
\setboolean{showedits}{false}
\usepackage[markup=underlined]{changes}
\definechangesauthor[color=bleudefrance]{EM}
\newcommand{\aem}[1]{
\ifthenelse{\boolean{showedits}}
{\added[id=EM]{#1}}
{\!#1\hspace{-4.75pt}}
}
\newcommand{\repem}[2]{
\ifthenelse{\boolean{showedits}}
{\replaced[id=EM]{#1}{#2}}
{\!#1\hspace{-4.75pt}}
}
\newcommand{\dem}[1]{
\ifthenelse{\boolean{showedits}}
{\deleted[id=EM]{#1}}
{}
}

\IEEEoverridecommandlockouts                              

\usepackage{epsfig, latexsym, times, bbm}
\usepackage{color}
\usepackage{colortbl}
\usepackage{array}
\usepackage{cite}

\usepackage{amsmath}
\usepackage{graphics}
\usepackage{graphicx}
\usepackage{amssymb}
\usepackage{url}
\usepackage{comment}

\usepackage{amsthm}

\usepackage{multirow}

\usepackage{bbding}
\usepackage{amsfonts}
\usepackage{cases}

\usepackage{algorithmic}
\usepackage{mathrsfs}
\usepackage{bm}
\usepackage{verbatim}
\usepackage{enumerate}

\usepackage{booktabs}
\usepackage[linesnumbered,ruled,longend]{algorithm2e}
\usepackage{textcomp}
\usepackage{subeqnarray}
\usepackage{dsfont,optidef}

\usepackage{subfig}

\def\ba{\begin{align}}
	\def\ea{\end{align}}
\newcommand{\beq}{\begin{equation}}
	\newcommand{\eeq}{\end{equation}}
\newcommand{\bq}{\begin{eqnarray}}
	\newcommand{\eq}{\end{eqnarray}}
\newcommand{\bqn}{\begin{eqnarray*}}
	\newcommand{\eqn}{\end{eqnarray*}}
\newcommand{\bee}{\begin{enumerate}}
	\newcommand{\eee}{\end{enumerate}}
\newcommand{\bi}{\begin{itemize}}
	\newcommand{\ei}{\end{itemize}}
\newcommand{\bseq}{\begin{subequations}}
\newcommand{\eseq}{\end{subequations}}

\newcommand{\PreserveBackslash}[1]{\let\temp=\\#1\let\\=\temp}
\newcolumntype{C}[1]{>{\PreserveBackslash\centering}p{#1}}
\newcolumntype{R}[1]{>{\PreserveBackslash\raggedleft}p{#1}}
\newcolumntype{L}[1]{>{\PreserveBackslash\raggedright}p{#1}}

\renewcommand\thepage{}

\allowdisplaybreaks

\newcommand{\R}{\mathbb{R}}
\newcommand{\norm}[1]{\|#1\|}
\newcommand{\dd}[1]{\frac{\partial}{\partial #1}}

\newtheorem*{remark}{Remark}
\newtheorem{theorem}{{Theorem}}
\newtheorem{lemma}[theorem]{{Lemma}}
\newtheorem{definition}{{Definition}}
\newtheorem{corollary}[theorem]{{Corollary}}
\newtheorem{proposition}[theorem]{{Proposition}}

\newtheorem{assumption}{{Assumption}}

\title{\Large \bf
Saddle Flow Dynamics: Observable Certificates and Separable Regularization
}

\author{Pengcheng You and Enrique Mallada
\thanks{Pengcheng You and Enrique Mallada are with the Department of Electrical and Computer Engineering, Johns Hopkins University, Baltimore, MD 21218, US.
        {\tt \{pcyou,mallada\}@jhu.edu}}%
\thanks{The work was supported by NSF through grants CNS 1544771, EPCN 1711188, AMPS 1736448, and CAREER 1752362.}
}
\usepackage{setspace}
\begin{document}

\maketitle
\thispagestyle{empty}
\pagestyle{empty}

\begin{abstract}

This paper proposes a certificate, rooted in observability, for asymptotic convergence of saddle flow dynamics of convex-concave functions to a saddle point. This observable certificate directly bridges the gap between the invariant set and the equilibrium set in a LaSalle argument, and generalizes conventional conditions such as strict convexity-concavity and proximal regularization.
We further build upon this certificate to propose a separable regularization method for saddle flow dynamics that makes minimal requirements on convexity-concavity and yet still guarantees asymptotic convergence to a saddle point. Our results generalize to saddle flow dynamics with projections on the vector field and have an immediate application as a distributed solution to linear programs.

\end{abstract}


\section{Introduction}

There has been increasing interest in studying optimization algorithms from a dynamical systems view-point as means to understand their stability \cite{holding2020stability_a,holding2020stability_b}, rate of convergence \cite{wibisono2016variational,francca2019gradient,mohammadi2020robustness}, and robustness \cite{mohammadi2020robustness,lessard2016analysis,richert2015robust,cherukuri2017role}. For example, in the basic case of gradient descent dynamics for unconstrained convex optimization, the objective function monotonically decreases along trajectories towards the optimum, naturally rendering a Lyapunov function~\cite{khalil2002nonlinear}. Such realization later on leads to multiple extensions, including finite-time convergence~\cite{cortes2006finite,garg2020fixed}, acceleration~\cite{mohammadi2020robustness,francca2019gradient}, and time-varying optimization~\cite{rahili2016distributed,fazlyab2017prediction}.

One prominent branch of this line of work is the study of saddle flow dynamics, i.e., dynamics in the gradient descent direction on a sub-set of variables and the gradient ascent direction on the rest. Designed for locating min-max saddle points, saddle flow dynamics are particularly suited
for solving constrained optimization problems via primal-dual methods \cite{cherukuri2016asymptotic}, and finding Nash equilibria of zero-sum games \cite{gharesifard2013distributed}, which lead to a broad application spectrum, including power systems \cite{zhao2014design}, communication networks \cite{chiang2007layering}, and cloud computing \cite{goldsztajn2019proximal}.

The study of saddle flow dynamics is rooted in the seminal work \cite{arrow1958studies}, which first considers asymptotic behavior of saddle flows in the context of primal-dual algorithms. This has received renewed attention over the last decade. 
In particular, conventional conditions for asymptotic convergence are re-validated with advanced analytical tools in more general setups. For instance, \cite{cherukuri2016asymptotic} revisits the condition of strict convexity-concavity in the case of discontinuous vector fields, using LaSalle's invariance principle for discontinuous Caratheodory systems.

More recent analysis has been centered around asymptotic convergence of saddle flow dynamics under weaker conditions. A set of literature establishes stronger results that require only local strong convexity-concavity \cite{cherukuri2017role}, convexity-linearity, or strong quasiconvexity-quasiconcavity \cite{cherukuri2017saddle}. 
Further, asymptotic oscillatory behavior of saddle flow dynamics for general convex-concave functions and their $\Omega$ limit sets are explicitly characterized in \cite{holding2020stability_a,holding2020stability_b}.

Another active topic is to circumvent the above conditions via regularization, which is an appealing approach to handling the Lagrangian of constrained convex/linear optimization. This includes various penalty terms on equality constraints or even projected inequality constraints \cite{cherukuri2017distributed,richert2015robust}, as well as the proximal method \cite{goldsztajn2019proximal,Goldsztajn2020proximal}. The rationale of regularization is further interpreted in \cite{yamashita2020passivity} in the frequency domain.
Despite the merit of regularization that relaxes conditions for convergence, the extra penalty terms commonly introduce couplings that may require additional computation and communication overheads in order to realize distributed implementation.

Our work complements the existing literature by first obtaining a sufficient certificate for asymptotic convergence of saddle flow dynamics of a convex-concave function. The underlying rationale still adopts the standard quadratic Lyapunov function and LaSalle's invariance principle, but directly connects the invariant set and the equilibrium set through the existence of a certificate with certain observability properties. In this way, the asymptotic convergence to an equilibrium, i.e., a saddle point, is immediately available.
Our observable certificate 
is weaker than conventional conditions, e.g., strict convexity-concavity and proximal regularization.

We further exploit this certificate to develop a novel separable regularization method where only minimal convexity-concavity is required to establish convergence, a significantly milder condition than most existing ones in the literature, that includes bi-linear saddle functions as a special case. More importantly, the introduced regularization terms are separable and local, thus consistently preserving any distributed structure that original systems may have.
The certificate can be generalized to accommodate projections on the vector field and, as a result, allows us to apply our separable regularization to designing distributed solvers for linear programs.

The remainder of the paper is organized as follows.
Section~\ref{sec:saddle-flows} introduces the problem formulation with basic definitions and assumptions, followed by the key results on asymptotic convergence of saddle flow dynamics in Section~\ref{sec:asymp_convergece}.
We further generalize the results to the projected cases in Section~\ref{sec:proj_saddle_flow_dynamics}.
Section~\ref{sec:simulation} provides simulation validations and Section~\ref{sec:conclusion} concludes.

\subsubsection*{Notation}
Let $\mathbb{R}$ and $\mathbb{R}_{\ge 0}$ be the sets of real and non-negative real numbers, respectively.
Given two vectors $x,y\in\R^n$, $x_i$ and $y_i$ denote their $i^{\rm th}$ entries, respectively; and $x \le y$ holds if $x_i \le y_i$ holds for $\forall i$. 
Given a continuously differentiable function $S(x,y)\in\mathcal C^1$ with $S:\R^n\times \R^m\rightarrow \R$, we use $\dd{x}S(x,y)\in \R^{1\times n}$ and $\dd{y}S(x,y)\in \R^{1\times m}$ to denote the partial derivatives with respect to $x$ and $y$, respectively. We further define $\nabla_x S(x,y)=\left[\dd{x}S(x,y)\right]^T$.

\section{Problem Formulation}\label{sec:saddle-flows}

We consider a function $S(x,y)$ with $S:\R^{n} \times \R^m \rightarrow\R$. Our goal is to study different dynamic laws that seek to converge to some saddle point $(x_\star, y_\star)$ of $S(x,y)$.
\begin{definition}[Saddle Point]\label{def:saddle-point}
A point $(x_\star,y_\star)\in\R^{n\times m}$ is a saddle point of a function $S(x,y)$ if 
\begin{equation}\label{eq:saddle-critical}
   \left\{
   \begin{aligned}
   \nabla_x S(x_\star,y_\star) & =0 \\
   \nabla_y S(x_\star,y_\star) & =0 
   \end{aligned} \right.
\end{equation}
holds, and $S(x_\star,y_\star)$ is not a local extremum.
\end{definition}

While in general, such questions could be asked on a setting without any further restrictions, neither the existence of saddle points nor convergence towards them is easy to guarantee.
For the purpose of this paper, we focus our attention on functions $S(x,y)$ that are \emph{convex-concave}.
\begin{definition}[Convex-Concave Functions]\label{def:convex-concave}
A function $S(x,y)$ is convex-concave if and only if $S(\cdot,y)$ is convex for $\forall y\in R^m$ and $S(x,\cdot)$ is concave for $\forall x\in\R^n$.
A function $S(x,y)$ is strictly convex-concave if and only if $S(x,y)$ is convex-concave and either $S(\cdot,y)$ is strictly convex for $\forall y\in R^m$ or $S(x,\cdot)$ is strictly concave for $\forall x\in\R^n$.
\end{definition}

Due to the convexity-concavity of $S(x,y)$, \eqref{eq:saddle-critical} implies that indeed any of its saddle points $(x_\star,y_\star)$ can be characterized by
\begin{equation}\label{eq:saddle-inequality}
    S(x_\star,y)\leq S(x_\star,y_\star) \leq S(x,y_\star)\,, 
\end{equation}
for $\forall x\in \R^n$ and $\forall y\in \R^m$.
Therefore, we are specifically interested in minimizing $S(x,y)$ over $x$ and meanwhile maximizing $S(x,y)$ over $y$.

Throughout this work, we will assume that such a point $(x_\star,y_\star)$ does exist and that $S(x,y)$ is continuously differentiable, i.e., $S(x,y)\in \mathcal C^1$, as formally summarize below.

\begin{assumption}\label{ass:paper-assumption}
$S(x,y)$ is convex-concave, continuously differentiable, and there exists at least one saddle point $(x_\star,y_\star)$ satisfying \eqref{eq:saddle-inequality}.
\end{assumption}

The continuous differentiability in Assumption \ref{ass:paper-assumption} is introduced to simplify the exposition. It does not significantly limit the scope of the results as one can always derive a continuously differentiable surrogate of a continuous convex/concave function by means of the Moreau Envelope \cite{parikh2014proximal}.


Given a convex-concave function $S(x,y)$ satisfying Assumption~\ref{ass:paper-assumption}, we refer to the following dynamic law
\begin{subequations}\label{eq:saddle-flow}
\begin{align}
    \dot x &= -\,\nabla_x S(x,y)\,,\label{eq:p-saddle-flow}\\
    \dot y &= +\,\nabla_y S(x,y)\,, \label{eq:d-saddle-flow}
\end{align}
\end{subequations}
as the saddle flow dynamics of $S(x,y)$. 
Due to convexity-concavity, a point is
an equilibrium of \eqref{eq:saddle-flow} if and only if it is a saddle point of $S(x,y)$, and the dynamic law drives the system towards such points in the directions of gradient descent and ascent, respectively, for $x$ and $y$. 
We will mainly work with this standard form of saddle flow dynamics to locate a saddle point of $S(x,y)$.


\section{Asymptotic Convergence}\label{sec:asymp_convergece}

In this section we present an observable certificate that guarantees asymptotic convergence of the saddle flows dynamics \eqref{eq:saddle-flow} to a saddle point of $S(x,y)$.
We show that two conventional conditions of strict convexity-concavity and proximal regularization satisfy this certificate as special cases.
We further build on this certificate to develop a separable regularization method that entails minimal convexity-concavity requirements on $S(x,y)$ for saddle flow dynamics to asymptotically converge to a saddle point. 

\subsection{Observable Certificates}\label{ssec:general_principle}

We now describe the proposed observable certificate for the saddle flow dynamics \eqref{eq:saddle-flow} to asymptotically converge to a saddle point of $S(x,y)$.
\begin{definition}[Observable Certificate]
\label{def:certificate}
A function $h(x,y)$ with $h:\mathbb{R}^n \times \mathbb{R}^m \rightarrow \mathbb{R}_{\ge0}^2$ is an observable certificate of $S(x,y)$, if and only if there exists a saddle point $(x_\star,y_\star)$ such that
\begin{equation}\label{eq:bounded-auxiliary-function}
    \begin{bmatrix}
    S(x_\star,y_\star) - S(x_\star,y)\\
    S(x,y_\star) - S(x_\star,y_\star) 
    \end{bmatrix}\ge h(x,y) \ge 0 
\end{equation}
holds and for any trajectory $(x(t),y(t))$ of \eqref{eq:saddle-flow} that satisfies $h(x(t),y(t)) \equiv 0$, we have $\dot x, \dot y \equiv 0$.
\end{definition}
\begin{remark}
We call $h(x,y)$ an observable certificate, due the second property of Definition \ref{def:certificate}, which is akin to \eqref{eq:saddle-flow} having $h(x,y)$ as an observable output. It is exactly this observability property that will allow us to connect invariant sets with saddle-points.
\end{remark}

\begin{assumption}~\label{ass:auxiliary-function}
$S(x,y)$ has an observable certificate $h(x,y)$ as given by Definition \ref{def:certificate}.
\end{assumption}
Checking whether Assumption~\ref{ass:auxiliary-function} holds basically requires hunting for a qualified observable certificate $h(x,y)$ of $S(x,y)$.
Under this assumption, asymptotic convergence of the saddle flow dynamics \eqref{eq:saddle-flow} is formally stated below.
\begin{theorem}[Sufficiency of Observable Certificates]\label{th:general-principle}
Let Assumptions~\ref{ass:paper-assumption} and \ref{ass:auxiliary-function} hold. Then the saddle flow dynamics \eqref{eq:saddle-flow} asymptotically converge to some saddle point $(x_\star,y_\star)$ of $S(x,y)$. 
\end{theorem}
\begin{proof}
The proof follows from applying LaSalle's invariance principle \cite{khalil2002nonlinear} to the following candidate Lyapunov function
\begin{equation}\label{eq:Vxy}
    V(x,y) = \frac{1}{2}\norm{x-x_\star}^2 + \frac{1}{2}\norm{y-y_\star}^2 \, ,
\end{equation}
where $(x_\star,y_\star)$ is the saddle point identified in Definition~\ref{def:certificate}.
Taking the Lie derivative of~\eqref{eq:Vxy} along the trajectory $(x(t),y(t))$ of \eqref{eq:saddle-flow} gives
\begin{align*}
    \dot V &=(x-x_\star)^T\dot x+(y-y_\star)^T\dot y\\
    &=(x-x_\star)^T\left[-\nabla_xS(x,y)\right]+(y-y_\star)^T\left[+\nabla_yS(x,y)\right]\\
    &=(x_\star-x)^T\nabla_xS(x,y)-(y_\star-y)^T\nabla_yS(x,y)\\
    &\leq
    S(x_\star,y)-S(x,y)-
    \left(S(x,y_\star)-S(x,y)\right)\\
    &=S(x_\star,y)-S(x,y_\star)
    \\
    &=\underbrace{S(x_\star,y) -S(x_\star,y_\star)}_{\leq0} + \underbrace{S(x_\star,y_\star) - S(x,y_\star)}_{\leq0} \, ,
\end{align*}
where the second equality plugs in \eqref{eq:saddle-flow}, the first inequality applies the convexity-concavity of $S(x,y)$, and the last inequality follows from the saddle property \eqref{eq:saddle-inequality} of $(x_\star,y_\star)$.

Since \eqref{eq:Vxy} is radially unbounded, every sub-level set of it is compact. From above, it follows that the trajectories of \eqref{eq:saddle-flow} are bounded and contained in an invariant domain
\beq
D_0(x_0,y_0):=\left\{ (x,y) \ \vert \  V(x,y)\leq V(x_0,y_0) \right\} \, ,
\eeq
where $(x_0,y_0)$ is any given initial point.
LaSalle's invariance principle then implies that any trajectory of \eqref{eq:saddle-flow} should converge to the largest invariant set 
\beq
\mathbb{S} := D_0 \cap\left\{ (x,y)  \ \vert \  \dot V(x(t),y(t)) \equiv 0   \right\} \, .
\eeq
Given Assumption~\ref{ass:auxiliary-function}, \eqref{eq:bounded-auxiliary-function} implies that $\mathbb{S}$ is indeed a subset of 
\begin{equation}
     \left\{ (x,y)  \ \vert \  h(x(t),y(t)) \equiv 0   \right\} \, ,
\end{equation}
which is further a subset of the equilibrium set of \eqref{eq:saddle-flow}, denoted as 
\begin{equation}
    \mathbb{E} := \left\{ (x,y)  \ \vert \  \dot x(t), \dot y(t)  \equiv 0   \right\} \, ,
\end{equation}
i.e., $\mathbb{S} \subset \mathbb{E}$.

It follows that the invariant set $\mathbb{S}$ contains only equilibrium points. If $\mathbb{S}$ were to be composed of isolated points -- only possible when there is a unique saddle point -- this would be sufficient to prove convergence to the (unique) saddle point. However, in general LaSalle's invariance principle only shows asymptotic convergence to the invariant set, without guaranteeing convergence to a point within it, even in the case where the set is composed of equilibrium points.

This issue is circumvented by the fact that all the equilibria within $\mathbb{S}$ are stable. See, e.g., \cite[Corollary~5.2]{bhat2003nontangency}.
Alternatively, notice that $\mathbb{S}$ is compact, and as a result any trajectory within the $\Omega$ limit set of \eqref{eq:saddle-flow} has a convergent sub-sequence. Let $(\bar x,\bar y)$ be the limit point of such a sequence. Due to $(\bar x,\bar y)\in \mathbb{S}$, it is also a saddle point. By changing $(x_\star,y_\star)$ specifically to $(\bar x,\bar y)$ in the definition of $V(x,y)$, it follows that $0\leq V(x(t),y(t))\rightarrow 0$ holds, which implies $(x(t),y(t))\rightarrow(\bar x,\bar y)$.
\end{proof}

The existence and characterization of such observable certificates $h(x,y)$ may still be vague from only Definition~\ref{def:certificate}. We next discuss how they can be identified through concrete examples.

\subsection{Revisiting Existing Conditions}\label{ssec:existing_conditions}

We show that the observable certificate is indeed a weaker condition underneath some of the conventional ones required for asymptotic convergence of the saddle flow dynamics \eqref{eq:saddle-flow}.

\subsubsection{Strict Convexity-Concavity}

The most common condition is arguably the strict convexity-concavity of $S(x,y)$ \cite{cherukuri2016asymptotic}.
We formalize its connection with our observable certificate as below.
\begin{assumption}\label{ass:strict-cc}
$S(x,y)$ is strictly convex-concave.
\end{assumption}
\begin{proposition}[Strict Convexity-Concavity]\label{prop:strict_convexity-concavity}
Let Assumptions~\ref{ass:paper-assumption} and \ref{ass:strict-cc} hold. Then the function
\beq
h(x,y) : =     \begin{bmatrix}
    S(x_\star,y_\star) - S(x_\star,y)\\
    S(x,y_\star) - S(x_\star,y_\star) 
    \end{bmatrix}  \, ,
\eeq
with $(x_\star,y_\star)$ being an arbitrary saddle point of $S(x,y)$, is an observable certificate of $S(x,y)$.
\end{proposition}
\begin{proof}
Indeed, the upper bound of $h(x,y)$ in \eqref{eq:bounded-auxiliary-function} itself is naturally a qualified observable certificate.
Without loss of generality, we assume $S(x,y)$ is only strictly convex in $x$ for $\forall y$ and just concave in $y$ for $\forall x$. Due to the strict convexity in $x$, $S(x,y_\star) \equiv S(x_\star,y_\star)$ implies $x(t) \equiv x_\star$ uniquely. Meanwhile, $\nabla_y S(x_\star,y) \equiv 0$ follows from $S(x_\star,y) \equiv S(x_\star,y_\star)$, which leads to $\dot y = \nabla_y S(x,y) \equiv \nabla_y S(x_\star,y) \equiv 0 $. Therefore, $\dot x,\dot y\equiv 0$ is guaranteed from $h(x,y)\equiv 0$.
\end{proof}
Asymptotic convergence of the saddle flow dynamics \eqref{eq:saddle-flow} then immediately follows from Theorem~\ref{th:general-principle}.

\begin{corollary}
Let Assumptions \ref{ass:paper-assumption} and \ref{ass:strict-cc} hold.
Then the saddle flow dynamics \eqref{eq:saddle-flow} asymptotically converge to some saddle point $(x_\star,y_\star)$ of $S(x,y)$.
\end{corollary}

\subsubsection{Proximal Regularization}

In the particular form of saddle flow dynamics known as primal-dual dynamics~\cite{cherukuri2016asymptotic}, a proximal regularization method is proposed in \cite{goldsztajn2019proximal,Goldsztajn2020proximal} to guarantee asymptotic convergence of the regularized saddle flow dynamics, even in the absence of strict convexity-concavity. 
Specifically, a surrogate differentiable convex-concave function $$\bar S(z,y):= \min_{x} \left\{ S(x,y) + \frac{1}{2}\Vert x-z \Vert^2 \right\} $$ is defined from $S(x,y)$ that maintains the same saddle points \cite{Goldsztajn2020proximal}.
Then the following regularized saddle flow dynamics 
\begin{subequations}\label{eq:saddle-flow_proximal}
\begin{align}
    \dot z &= -\,\nabla_z \bar S(z,y)\,,\label{eq:p-saddle-flow_proximal}\\
    \dot y &= +\,\nabla_y \bar S(z,y)\,, \label{eq:d-saddle-flow_proximal}
\end{align}
\end{subequations}
suffice to locate a saddle point.
We formalize the connection of this method with our observable certificate as follows.
\begin{proposition}[Proximal Regularization]\label{prop:proximal-convergence}
Let $S(x,y)$ be a Lagrangian function for some constrained convex program and Assumption~\ref{ass:paper-assumption} hold. Then the function
\beq\label{eq:auxiliary_function_proximal}
h(z,y) : =     \begin{bmatrix}
    \bar S(z_\star,y_\star) - \bar S(z_\star,y)\\
    \frac{1}{2} \Vert \bar x(z,y_\star) - z\Vert^2 
    \end{bmatrix} \, ,
\eeq
with $\bar x(z,y_\star) : = \arg \min_{x} \left\{S(x,y_\star) + \frac{1}{2}\Vert x-z\Vert^2 \right\}$ and $(z_\star,y_\star)$ being an arbitrary saddle point of $\bar S(z,y)$, is an observable certificate of $\bar S(z,y)$.
\end{proposition}

Details of the proof are omitted here and readers are referred to \cite{Goldsztajn2020proximal} for more insights. We remark that the identification of this observable certificate \eqref{eq:auxiliary_function_proximal} does not significantly alleviate the analysis overheads since the complementary equilibrium properties of proximal regularization on the Lagrangian $S(x,y)$ are still crucial to validating the observable certificate \eqref{eq:auxiliary_function_proximal} and establishing convergence.

Anyhow, the existence of an observable certificate satisfies Assumption~\ref{ass:auxiliary-function} for $\bar S(z,y)$ and thus asymptotic convergence of the saddle flow dynamics \eqref{eq:saddle-flow_proximal} follows immediately from Theorem~\ref{th:general-principle}.
\begin{corollary}
Let $S(x,y)$ be a Lagrangian function for some constrained convex program and Assumption~\ref{ass:paper-assumption} hold. Then the regularized saddle flow dynamics \eqref{eq:saddle-flow_proximal} asymptotically converge to some saddle point $(z_\star,y_\star)$ of $\bar S(z,y)$, with $(x_\star =z_\star,y_\star)$ being a saddle point of $S(x,y)$.
\end{corollary}
In fact, even the differentiability in Assumption~\ref{ass:paper-assumption} is not required since the surrogate $\bar S(z,y)$ can be continuously differentiable regardless. However, this proximal regularization method is only limited to the particular form of primal-dual dynamics.

\subsection{Separable Regularization}\label{ssec:augmentation_regularization}






    

We further design a novel separable regularization method that exploits our observable certificate and only requires Assumption~\ref{ass:paper-assumption} for a regularized version of saddle flow dynamics to asymptotically converge to a saddle point.
The key of this method is to augment the domain of $S(x,y)$ and introduce regularization terms without altering the positions of the original saddle points.
In particular, we propose a regularized surrogate for $S(x,y)$ via the following augmentation
\begin{equation}\label{eq:regularized-saddle}
    S(x,z, y,w): = \frac{1}{2\rho  }\|x-z\|^2 + S(x,y) - \frac{1}{2\rho}\|y-w \|^2,
\end{equation}
where $z\in \R^n$ and $w \in \R^m$ serve as two new sets of virtual variables and $\rho >0$ is a constant regularization coefficient. 
It is straightforward to verify the fixed positions of saddle points between $S(x,y)$ and $S(x,z,y,w)$ with virtual variables aligned with original variables.

\begin{lemma}[Saddle Point Invariance]\label{th:saddle-characterization}
Let Assumption~\ref{ass:paper-assumption} hold. Then a point $(x_\star,y_\star)$ is a saddle point of $S(x,y)$ if and only if $(x_\star, z_\star, y_\star,w_\star)$ is a saddle point of $S(x,z,y,w)$, with
\begin{equation}
    x_\star=z_\star ~\textrm{ and }~  y_\star=w_\star  \, . \label{eq:reg-saddle-point-property}
\end{equation}
\end{lemma}
\begin{proof}
Recall the saddle property \eqref{eq:saddle-inequality} of a saddle point, this theorem follows immediately from
\begin{small}
\begin{align*}
&S(x_\star,z_\star,y,w)\leq S(x_\star,x_\star,y_\star,y_\star)\leq S(x,z,y_\star,w_\star) \\[1.5ex]
\iff &S(x_\star,z_\star,y,w)\leq S(x_\star,y_\star)\leq S(x,z,y_\star,w_\star) \\ 
\iff & S(x_\star,y)\!-\!\frac{\norm{y\!-\!w}}{2\rho}^2\leq S(x_\star,y_\star)\leq S(x,y_\star)\!+\!\frac{\norm{x\!-\!z}^2}{2\rho}\\
\iff & S(x_\star,y)\leq S(x_\star,y_\star)\leq S(x,y_\star)  \, ,
\end{align*}\end{small}%
where the first and second steps build upon the definition \eqref{eq:regularized-saddle} of $S(x,z,y,w)$, and the third step uses norm non-negativity.
\end{proof}
The regularized function $S(x,z,y,w)$ is convex in $(x,z)$, concave in $(y,w)$, and continuously differentiable with at least one saddle point, by its definition in \eqref{eq:regularized-saddle} and Lemma~\ref{th:saddle-characterization}. Therefore, Assumption~\ref{ass:paper-assumption} also holds for $S(x,z,y,w)$.
Lemma~\ref{th:saddle-characterization} ensures that whenever we locate a saddle point of $S(x,z,y,w)$, a saddle point of $S(x,y)$ satisfying \eqref{eq:saddle-inequality} is attained simultaneously.
This motivates us to instead look at the saddle flow dynamics of $S(x,z,y,w)$.

Following \eqref{eq:saddle-flow}, this regularized version of saddle flow dynamics are given by 
\begin{subequations}\label{eq:reg-saddle-flow}
\begin{align}
    \dot x &
    =-\,\nabla_xS(x, y)-\frac{1}{\rho}(x-z)
    \,,\label{eq:reg-saddle-x}\\
    \dot{z} &
    =\frac{1}{\rho}(x-z)
    \,,\label{eq:reg-saddle-z}\\
    \dot y &
    =+\,\nabla_y S(x,y)-\frac{1}{\rho}(y-w)
    \,,
    \label{eq:reg-saddle-y}\\
    \dot{w} &
    = \frac{1}{\rho}(y-w)\,.
    \label{eq:reg-saddle-w}
\end{align}
\end{subequations}
Although this dynamic law has twice as many state variables as its prototype \eqref{eq:saddle-flow}, it is important to notice that, unlike the proximal gradient algorithm \cite{goldsztajn2019proximal,Goldsztajn2020proximal,parikh2014proximal} and the equality constrained regularization \cite{richert2015robust,cherukuri2017distributed}, \eqref{eq:reg-saddle-flow} still preserves the same distributed structure that \eqref{eq:saddle-flow} may have. As a result, it can be implemented in a fully distributed fashion.

We are now ready to provide the key result that the regularized saddle flow dynamics \eqref{eq:reg-saddle-flow} asymptotically reach a saddle point as long as the minimal convexity-concavity holds for $S(x,y)$.

\begin{proposition}[Separable Regularization]\label{th:regularization-asymptotic-convergence}
Let Assumption~\ref{ass:paper-assumption} hold. Then the function
\beq
h(x,z,y,w):=
\begin{bmatrix}
     \frac{1}{2\rho}\Vert y-w\Vert^2 \\
     \frac{1}{2\rho} \Vert x-z \Vert^2
\end{bmatrix}   \, 
\eeq
is an observable certificate of $S(x,z,y,w)$.
\end{proposition}
\begin{proof}
The above observable certificate $h(x,z,y,w)$ satisfies \eqref{eq:bounded-auxiliary-function} in light of the following calculation:
\begin{align*}
  &  \begin{bmatrix}
    S(x_\star,z_\star,y_\star,w_\star) - S(x_\star,z_\star,y,w) \\
    S(x,z,y_\star,w_\star) - S(x_\star,z_\star,y_\star,w_\star) 
    \end{bmatrix}\\
   \ge &  \begin{bmatrix}
    \underbrace{S(x_\star,y_\star) - S(x_\star,y) }_{\ge 0}+ \frac{1}{2\rho}\Vert y-w\Vert^2 \\
    \underbrace{S(x,y_\star)    - S(x_\star,y_\star) }_{\ge 0}+ \frac{1}{2\rho} \Vert x-z \Vert^2
    \end{bmatrix}\\
   \ge &  \begin{bmatrix}
     \frac{1}{2\rho}\Vert y-w\Vert^2 \\
     \frac{1}{2\rho} \Vert x-z \Vert^2
    \end{bmatrix}\\
    \ge & \ 0 \, .
\end{align*}
The fact that $h(x,z,y,w) \equiv 0$ implies $x(t) \equiv z(t)$ and $y(t)\equiv w(t)$ enforces $\dot z, \dot w \equiv 0$ according to \eqref{eq:reg-saddle-z}, \eqref{eq:reg-saddle-w}, and then $\dot x, \dot y \equiv 0$ is simultaneously guaranteed. 
\end{proof}
Assumption~\ref{ass:auxiliary-function} holds for the regularized function $S(x,z,y,w)$ and asymptotic convergence of the regularized saddle flow dynamics \eqref{eq:reg-saddle-flow} follows immediately from Theorem~\ref{th:general-principle}.

\begin{corollary}
Let Assumption~\ref{ass:paper-assumption} hold. Then the regularized saddle flow dynamics \eqref{eq:reg-saddle-flow} asymptotically converge to some saddle point $(x_\star,z_\star,y_\star,w_\star)$ of $S(x,z,y,w)$, with $(x_\star,y_\star)$ being a saddle point of $S(x,y)$.
\end{corollary}

\begin{remark}
Proposition~\ref{th:regularization-asymptotic-convergence} indicates that only the convexity-concavity of $S(x,y)$ is required to asymptotically arrive at a saddle point through the regularized saddle flow dynamics \eqref{eq:reg-saddle-flow}. 
This condition is significantly milder than most existing ones in the literature, and is in some sense \emph{minimal}, as it includes bi-linear saddle functions as a special case. Unlike the aforementioned proximal regularization method in Section~\ref{ssec:existing_conditions}, our separable regularization method applies to saddle flow dynamics of general convex-concave functions.
\end{remark}

\section{Projected Saddle Flow Dynamics} \label{sec:proj_saddle_flow_dynamics}

In this section we generalize the results in Section~\ref{sec:asymp_convergece} to account for projections on the vector field of the saddle flow dynamics \eqref{eq:saddle-flow} that are commonly introduced in the case of solving inequality constrained optimization problems.

Specifically, we look at a projected version of saddle flow dynamics of a convex-concave function $S(x,y)$ as below:
\begin{subequations}\label{eq:saddle_flow_proj}
\begin{align}
    \dot x & = -\nabla_x S(x,y), \\
    \dot y & = \left[+\nabla_y S(x,y) \right]^+_y,
\end{align}
\end{subequations}
where, without loss of generality, we define the element-wise projection $[\cdot]_y^+$ only on part of the vector field regarding $y$ as 
\begin{equation}
    \left[\nabla_{y_i} S(x,y) \right]^+_{y_i}: = 
    \left\{
    \begin{aligned}
    & \nabla_{y_i} S(x,y),   \qquad\qquad \quad \ \textrm{if~}y_i > 0, \\
    & \max\left\{\nabla_{y_i} S(x,y), 0 \right\},  \quad \textrm{otherwise}.
    \end{aligned}\right.
\end{equation}
With this projection, $y(t)$ is constrained to be non-negative as long as it starts with a non-negative initial point. Accordingly, we slightly modify Assumption~\ref{ass:paper-assumption} to guarantee the existence of such saddle points.
\begin{assumption}\label{ass:paper-assumption_proj}
$S(x,y)$ is convex-concave, continuously differentiable, and there exists at least one saddle point $(x_\star,y_\star \ge 0)$ satisfying \eqref{eq:saddle-inequality}.
\end{assumption}
\noindent
In this context, saddle points are restrained to ones in the non-negative orthant of $y$. Therefore, any observable certificate of $S(x,y)$ will be defined on a saddle point $(x_\star,y_\star \ge 0)$ in Definition~\ref{def:certificate}.
By symmetry, an analogous projection may also be imposed on the other part of the vector field regarding $x$.
Next we formally generalize the sufficiency of observable certificates developed in Section~\ref{ssec:general_principle}.

\subsection{Observable Certificates for Projected Flows}

The generalization of Theorem~\ref{th:general-principle} for asymptotic convergence of the projected saddle flow dynamics \eqref{eq:saddle_flow_proj} to a saddle point of $S(x,y)$ is summarized as follows.
\begin{theorem}[Sufficiency of Observable Certificates for Projected Flows]\label{th:general-principle_proj}
Let Assumptions~\ref{ass:auxiliary-function} and \ref{ass:paper-assumption_proj} hold. Then the projected saddle flow dynamics \eqref{eq:saddle_flow_proj} asymptotically converge to some saddle point $(x_\star,y_\star \ge 0)$ of $S(x,y)$.
\end{theorem}
The proof requires a lemma regarding the projection $[\cdot]^+_y$.
\begin{lemma}\label{lm:projection_property}
Given any arbitrary $y,y_\star \in \mathbb{R}^m_{\ge 0}$ and $\nu \in \mathbb{R}^m$, 
$$\left( y - y_\star \right)^T \left( \left[\nu \right]^+_y - \nu  \right) \le 0 $$
holds.
\end{lemma}
\begin{proof}
The proof of this lemma follows from the fact that element-wise, $[\nu_i]^+_{y_i}$ differs from $\nu_i$ only in the case of $\nu_i < 0$ and $y_i=0$ where the projection is active, which implies
$$ \left( y - y_\star \right)^T \left( \left[\nu \right]^+_y - \nu  \right) = \sum_{i:\nu_i < 0, y_i=0} (0-y_{i*})(0-\nu_i)\le 0 \, . $$
\end{proof}
Using this lemma, the proof of Theorem~\ref{th:general-principle_proj} essentially follows from that of Theorem~\ref{th:general-principle} as follows.
\begin{proof}
Consider the same quadratic Lyapunov function \eqref{eq:Vxy}. Taking its Lie derivative along the trajectory $(x(t),y(t))$ of \eqref{eq:saddle_flow_proj} yields
\begin{align*}
    \dot V &=(x-x_\star)^T\dot x+(y-y_\star)^T\dot y\\
    &=(x-x_\star)^T\left[-\nabla_xS(x,y)\right]+(y-y_\star)^T\left[+\nabla_yS(x,y)\right]^+_y\\
    &=(x_\star-x)^T\nabla_xS(x,y)-(y_\star-y)^T\nabla_yS(x,y) \\
    & \quad \;  + \underbrace{ (y-y_\star)^T \left( \left[\nabla_yS(x,y)\right]^+_y-   \nabla_yS(x,y) \right) }_{\leq 0}\\
    &\leq
    S(x_\star,y)-S(x,y)-
    (S(x,y_\star)-S(x,y))\\
    &=S(x_\star,y)-S(x,y_\star)
    \\
    &=\underbrace{S(x_\star,y) -S(x_\star,y_\star)}_{\leq0} + \underbrace{S(x_\star,y_\star) - S(x,y_\star)}_{\leq0} \, ,
\end{align*}
where the key step is to use Lemma~\ref{lm:projection_property} in the first inequality.
The rest of the proof remains almost the same except that the largest invariant set is defined between the on-off switches of the projection.  
From above, $\dot V(x,y) \equiv 0$ additionally implies $$y(t)\equiv y_\star$$ or $$ \left[\nabla_y S(x,y)\right]^+_y \equiv   \nabla_y S(x,y) \, , $$
and an invariance principle for Caratheodory systems \cite{bacciotti2006nonpathological} can be applied to account for the discontinuities in the vector field due to the projection.
\end{proof}


\begin{figure*}[htp]
    \centering
    \hspace*{-0.4cm}
    \subfloat[Control inputs.]{
        \includegraphics[width=0.27\textwidth]{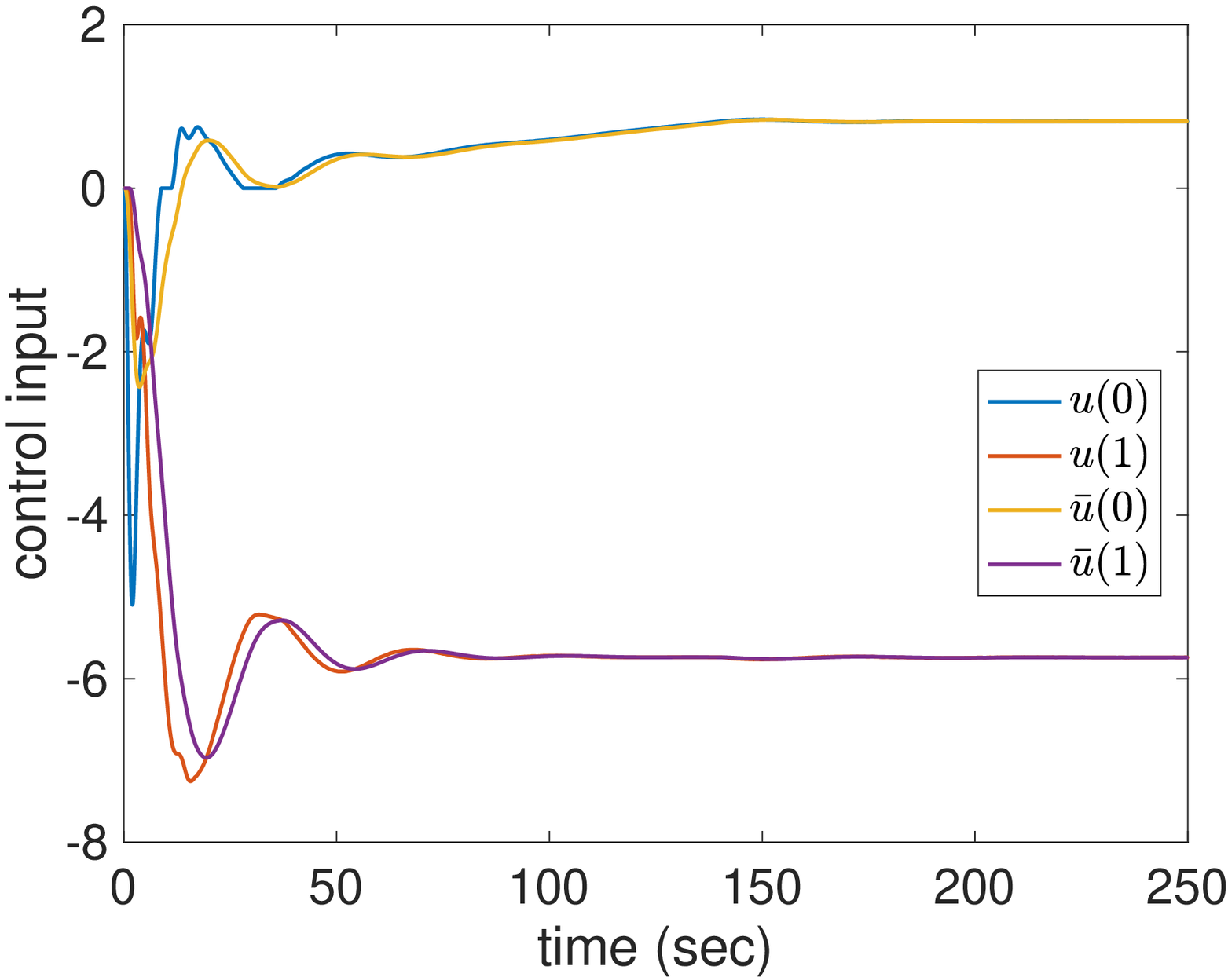}
    \label{fig:control_input}
    }    \hspace*{-0.8cm}
    \subfloat[State variables.]{
        \includegraphics[width=0.27\textwidth]{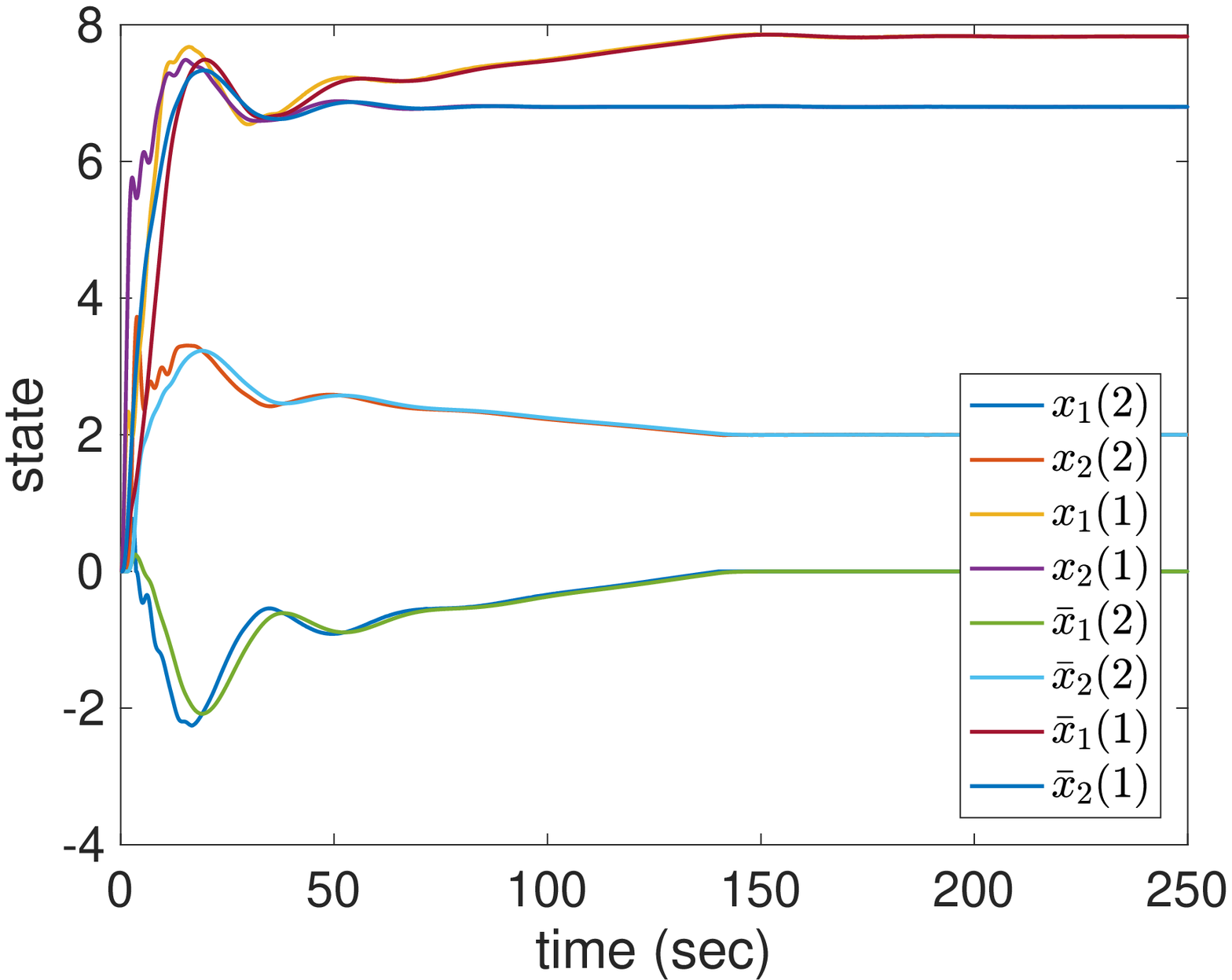}
    \label{fig:state}
    }\hspace*{-0.6cm}
    \subfloat[Dual variables.]{
        \includegraphics[width=0.27\textwidth]{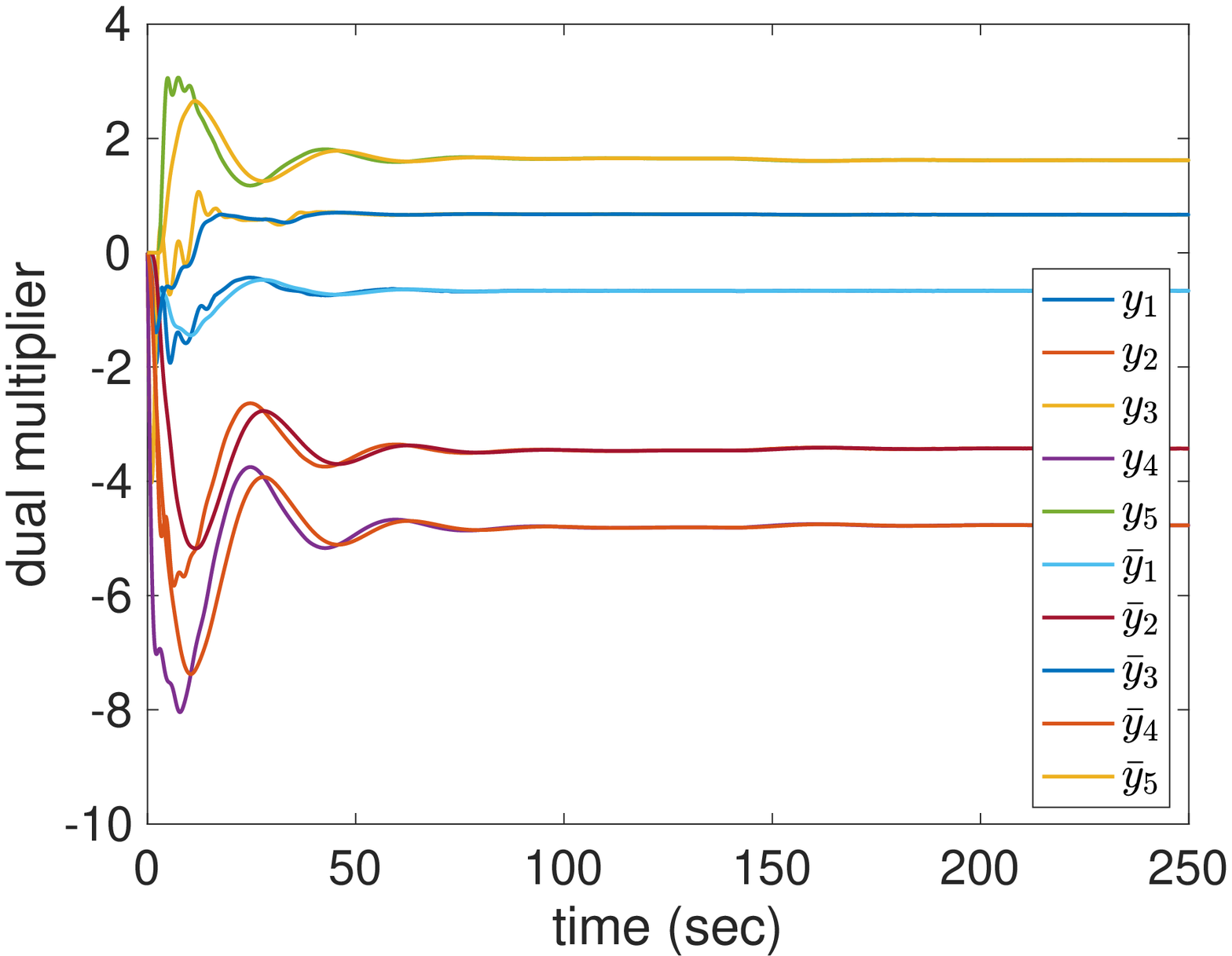}
    \label{fig:dual}
    }\hspace*{-0.7cm}
    \subfloat[System evolution.]{
        \includegraphics[width=0.27\textwidth]{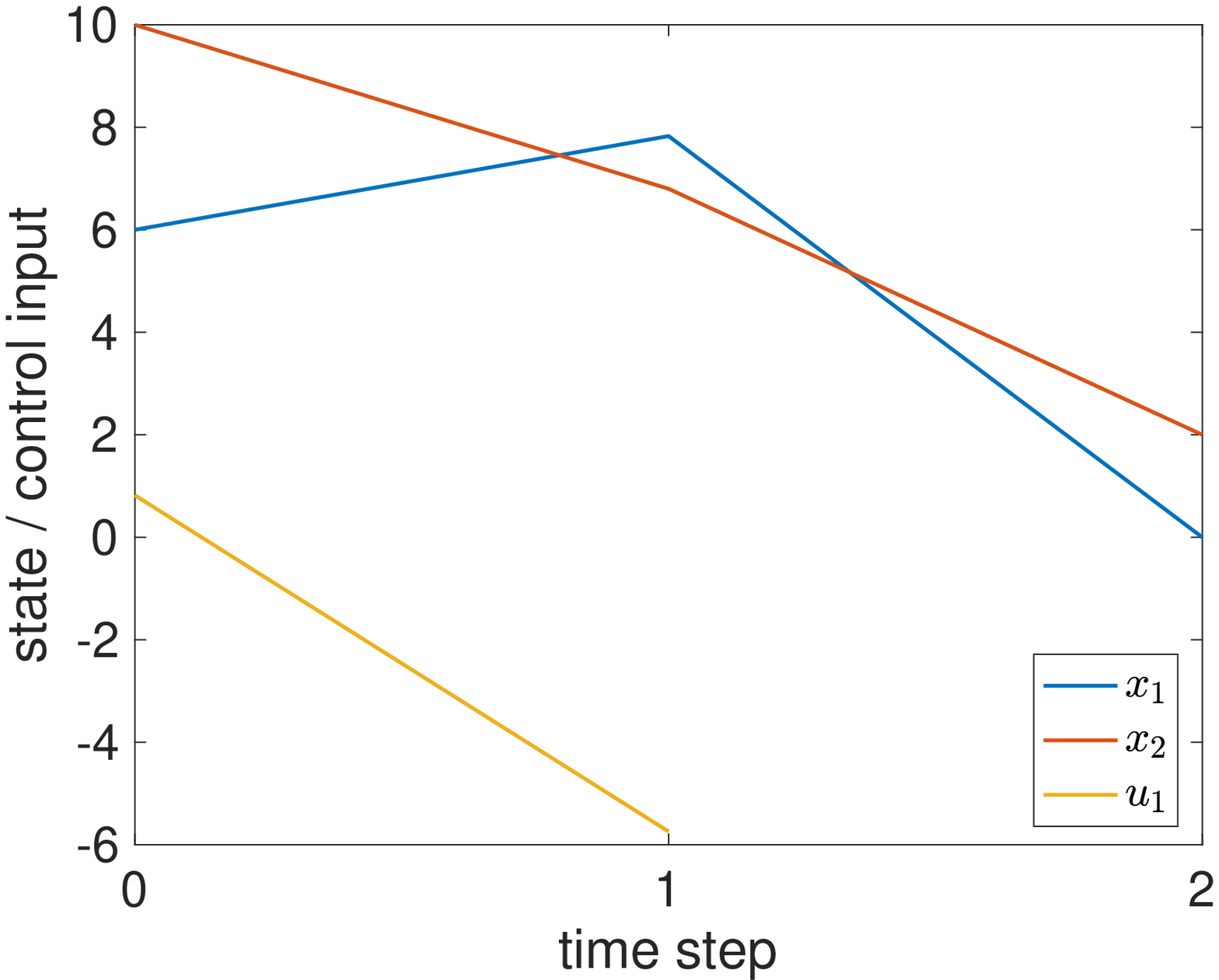}
    \label{fig:evolution}
    }
    \caption{A two-state linear optimal control problem solved by saddle flow dynamics with separable regularization.}
    \label{fig:optimal_control}
\end{figure*}


\subsection{Application: Distributed Solution to Linear Program}

Theorem~\ref{th:general-principle_proj} enables the separable regularization method in Section~\ref{ssec:augmentation_regularization} to apply to projected saddle flow dynamics as well since we can still identify the same observable certificate
\[ h(x,z,y,w):=
\begin{bmatrix}
     \frac{1}{2\rho}\Vert y-w\Vert^2 \\
     \frac{1}{2\rho} \Vert x-z \Vert^2
\end{bmatrix}  
\]
to satisfy Assumption~\ref{ass:auxiliary-function}. 
One of its straightforward applications involves solving inequality constrained linear programs in a distributed fashion with guaranteed asymptotic convergence to an optimal solution. 

Consider the following problem: 
\begin{subequations}\label{eq:constrainedLP}
\begin{eqnarray}
\min_{x\in\mathbb{R}^n} && c^T x \\
\mathrm{s.t.} && Ax-b \le 0 \ : \ y \in\mathbb{R}^m_{\ge0}
\end{eqnarray}
\end{subequations}
which corresponds to a bi-linear Lagrangian
$$S(x,y):=c^Tx + y^T(Ax-b).$$
We introduce virtual variables $z\in\mathbb{R}^n$, $w\in\mathbb{R}^m$ and a constant $\rho >0$ to define 
$$ S(x,z,y,w):= \frac{1}{2\rho}\Vert x-z \Vert^2 + c^Tx + y^T(Ax-b) - \frac{1}{2\rho}\Vert y-w \Vert^2 $$
to be its augmented Lagrangian. 
Lemma~\ref{th:saddle-characterization} implies that
$(x_\star,y_\star\ge 0)$ is a saddle point of $S(x,y)$, i.e., one optimal solution to \eqref{eq:constrainedLP}, if and only if
$(x_\star,z_\star=x_\star,y_\star\ge 0,w_\star=y_\star)$ is a saddle point of $S(x,z,y,w)$.

Then an algorithm to optimally solve a linear program of the form \eqref{eq:constrainedLP} follows immediately from asymptotic convergence of the following projected and regularized saddle flow dynamics:
\begin{subequations}\label{eq:reg-saddle-flow_proj}
\begin{align}
    \dot x &
    =-c-A^Ty-\frac{1}{\rho}(x-z)
    \,,\label{eq:reg-saddle-x_proj}\\
    \dot{z} &
    =\frac{1}{\rho}(x-z)
    \,,\label{eq:reg-saddle-z_proj}\\
    \dot y &
    = \left[Ax-b -\frac{1}{\rho}(y-w) \right]_y^+
    \,,
    \label{eq:reg-saddle-y_proj}\\
    \dot{w} &
    = \frac{1}{\rho}(y-w)\,,
    \label{eq:reg-saddle-w_proj}
\end{align}
\end{subequations}
which maintains the distributed structure where each agent $i=1,2,\dots,n$ may locally manage
\begin{subequations}\label{eq:reg-saddle-distributed}
\begin{align}
    \dot x_i &
    =-c_i-A_i^T y-\frac{1}{\rho}(x_i-z_i)
    \,,\\
    \dot{z_i} &
    =\frac{1}{\rho}(x_i-z_i)
    \,,
\end{align}
and/or each dual agent $j=1,2,\dots,m$ may locally manage
\begin{align}
    \dot y_j &
    = \left[A_j x-b_j -\frac{1}{\rho}(y_j-w_j) \right]_{y_j}^+
    \,,
    \\
    \dot{w_j} &
    = \frac{1}{\rho}(y_j-w_j)\,,
\end{align}
\end{subequations}
with $A_i$ and $A_j$ being the $i^{\text{th}}$ column and the $j^{\text{th}}$ row of $A$, respectively.

\section{Simulation Results}\label{sec:simulation}


We illustrate asymptotic convergence of the distributed algorithm \eqref{eq:reg-saddle-distributed} for linear programs, as guaranteed by our observable certificate, through a finite-horizon optimal control problem over a group of agents with coupled dynamics \cite{richert2015robust}. The goal is to use minimal aggregate control effort to maintain small system states subject to certain final performance requirements. In particular, consider the following problem defined on the time horizon $\mathcal{T}:=\{0,1,\dots,T\}$ and the set $\mathcal{N}:=\{1,2,\dots, N\}$ of agents:
\begin{subequations}
\begin{eqnarray}
\min && \sum_{t\in\mathcal{T}} \Vert x(t+1) \Vert_1 + \Vert u(t) \Vert_1 \\
\textrm{s.t.} && x(t+1) =  Gx(t) + Hu(t), ~ t\in\mathcal{T} \\
\label{eq:optimal_control.c}
&&  Dx(T+1)  - d  \le 0
\end{eqnarray}
\end{subequations}
Here $x(t)\in\R^{N}$ and $u(t)\in \R^{N}$ are respectively vectors of system states and control inputs at time $t$. $G\in\R^{N\times N}$ and $H\in\R^{N\times N}$ are the time-invariant dynamics matrix and control matrix, respectively. \eqref{eq:optimal_control.c} specifies the constraint on the final state with given $D\in\R^{M\times N}$ and $d\in\R^{M}$.
The initial state $x(0)$ is fixed.
This problem can be cast into the standard form of a linear program by splitting each variable of state and control input into two, representing its positive and negative components, e.g., $x(t) = x^+(t) - x^-(t)$ with $x^+(t),x^-(t)\ge0$:
\begin{small}
\begin{eqnarray*}
\min && \sum_{t\in\mathcal{T}} \sum_{n\in\mathcal{N}} x^+_n(t+1)  + x^-_n(t+1)+ u^+_n(t)+ u^-_n(t)  \\
\textrm{s.t.} && x^+(t+1) - x^-(t+1) =  G\left(x^+(t) - x^-(t)\right) \\
&& \qquad \qquad \qquad \quad \, + H\left(u^+(t) - u^- (t)\right), ~  t\in\mathcal{T} \\
&& D\left(x^+(T+1) - x^-(T+1)\right)  - d  \le 0 \\
&& x^+(t+1),x^-(t+1),u^+(t),u^-(t)\ge0, ~ t\in\mathcal{T} 
\end{eqnarray*}\end{small}%
It can be verified that at the optimum the split twin components of any variable cannot be simultaneously positive \cite{dantzig2006linear}. We apply \eqref{eq:reg-saddle-distributed} to solve this linear program with dual variables $y\in\R^{NT+M}$ and virtual variables $\bar x(t) \in \R^{N}$, $\bar u(t) \in \R^{N}$, $\bar y \in \R^{NT+M}$. Each agent $n$ is responsible for $4T$ variables, $2T$ for $x_n$ and $2T$ for $u_n$, while up to $4NT$ variables with any consensus-based distributed algorithm. Note that the dual variable updates may also be locally managed by up to $NT+M$ dual agents, if necessary.

We test a specific case of $N=2$ agents in a time horizon of $T=2$ time slots with the initial state $x(0)=[6~10]^T$. The system is driven by a single control input on the first agent's state. More specifically, we set the regularization coefficient $\rho$ to be 3, and the problem parameters are given by
\begin{small}
\begin{equation*}
G:=\begin{bmatrix}
     1.1 & 0 \\
     -0.7 & 1.1
\end{bmatrix}, \,
H:=\begin{bmatrix}
     1.5 & 0\\
     0 & 0
\end{bmatrix}, \,
D:=\begin{bmatrix}
     1 \\ 1.5
\end{bmatrix}^T, \,
d:= 3.
\end{equation*}
\end{small}%
The evolution of the continuous-time saddle flow dynamics for the control inputs, system states, dual multipliers as well as their virtual counterparts is displayed in Figures~\ref{fig:optimal_control}\subref{fig:control_input}-\ref{fig:optimal_control}\subref{fig:dual}, respectively. For each primal variable, we have already combined its two split components back together for conciseness.

Apparently, the regularization renders a variable bound tightly with its virtual counterpart in transient, and they asymptotically converge to the same value of an optimal solution in the limit. Notably, $y_5$, corresponding to \eqref{eq:optimal_control.c}, in Figure~\ref{fig:optimal_control}\subref{fig:dual} is the only dual variable subject to the projection that strictly stays non-negative.
Once all the variables are determined, the control inputs can be implemented to attain the system evolution with the minimal aggregate control effort and system states (in the sense of 1-norm), as depicted in Figure~\ref{fig:optimal_control}\subref{fig:evolution}. This is consistent with the optimal solution acquired from any centralized linear programming solver:
\begin{small}
\begin{equation*}
\begin{aligned}
& u_\star(0) = \begin{bmatrix}
     0.8190\\
     0
\end{bmatrix}, \
u_\star(1) = \begin{bmatrix}
     -5.7410\\
     0
\end{bmatrix}, \\
& x_\star(1) = \begin{bmatrix}
      7.8286 \\
    6.8000
\end{bmatrix}, \
x_\star(2) = \begin{bmatrix}
           0 \\
    2.0000
\end{bmatrix}.
\end{aligned}
\end{equation*}
\end{small}%


\section{Conclusion}\label{sec:conclusion}

This paper proposes an observable certificate that directly establishes connection between the invariant set and the equilibrium set for saddle flow dynamics of a convex-concave function such that the asymptotic convergence to a saddle point can be guaranteed. The certificate is rooted in observability, and we identify the existence of such observable certificates in the presence of conventional conditions, e.g., strict convexity-concavity and proximal regularization, as well as the proposed separable regularization method. Therefore, our observable certificate is a weaker condition, and it further generalizes to situations with projections on the vector field of saddle flow dynamics.
Besides, the novel separable regularization method that builds on our observable certificate requires only minimal convexity-concavity to establish convergence and enjoys a separable structure for potential distributed implementation, which is demonstrated through an application to distributed linear programming.






\bibliographystyle{IEEEtran} 
\bibliography{refs.bib} 

\end{document}